\documentclass[11pt]{article}%
\usepackage{amsmath}
\usepackage{amsfonts}
\usepackage{amssymb}
\usepackage{graphicx}
\usepackage{hyperref}%
\newtheorem{theorem}{Theorem}[section]

\newtheorem{corollary}[theorem]{Corollary}

\newtheorem{example}[theorem]{Example}

\newtheorem{lemma}[theorem]{Lemma}

\newtheorem{proposition}[theorem]{Proposition}

\newenvironment{proof}[1][Proof]{\noindent\textbf{#1.} }{\ \rule{0.5em}{0.5em}}
\begin{document}

\title{Lattice norms on the unitization of a truncated normed Riesz space}
\author{Karim Boulabiar\thanks{Corresponding author:
\texttt{karim.boulabiar@ipest.rnu.tn}}\quad and\quad Hamza Hafsi\\{\small Laboratoire de Recherche LATAO}\\{\small D\'{e}partement de Mathematiques, Facult\'{e} des Sciences de Tunis}\\{\small Universit\'{e} de Tunis El Manar, 2092, El Manar, Tunisia}}
\date{\textit{Dedicated to the memory of Coenraad Labuschagne}}
\maketitle

\begin{abstract}
Truncated Riesz spaces was first introduced by Fremlin in the context of
real-valued functions. An appropriate axiomatization of the concept was given
by Ball. Keeping only the first Ball's Axiom (among three) as a definition of
truncated Riesz spaces, the first named author and El Adeb proved that if $E$
is truncated Riesz space then $E\oplus\mathbb{R}$ can be equipped with a
non-standard structure of Riesz space such that $E$ becomes a Riesz subspace
of $E\oplus\mathbb{R}$ and the truncation of $E$ is provided by meet with $1$.
In the present paper, we assume that the truncated Riesz space $E$ has a
lattice norm $\left\Vert .\right\Vert $ and we give a necessary and sufficient
condition for $E\oplus\mathbb{R}$ to have a lattice norm extending $\left\Vert
.\right\Vert $. Moreover, we show that under this condition, the set of all
lattice norms on $E\oplus\mathbb{R}$ extending $\left\Vert .\right\Vert $ has
essentially a largest element $\left\Vert .\right\Vert _{1}$ and a smallest
element $\left\Vert .\right\Vert _{0}$. Also, it turns out that any
alternative lattice norm on $E\oplus\mathbb{R}$ is either equivalent to
$\left\Vert .\right\Vert _{1}$ or equals $\left\Vert .\right\Vert _{0}$. As
consequences, we show that $E\oplus\mathbb{R}$ is a Banach lattice if and only
if $E$ is a Banach lattice and we get a representation's theorem sustained by
the celebrate Kakutani's Representation Theorem.

\end{abstract}



\section{Introduction and some preliminaries}

Truncated Riesz spaces has been defined by Fremlin \cite{F1974} as Riesz
subspaces of $\mathbb{R}^{X}$ satisfying \textsl{Stone's condition}, i.e.,
containing with any non-negative function $x$ its meet $1\wedge x$ with the
constant function $1$. This concept is fundamental to analysis and mainly to
measure theory \cite{F1974,JF1987}. Quite recently, Ball \cite{B2014} provided
an appropriate axiomatization of truncated Riesz spaces. Actually, Ball's
definition deals with Riesz spaces over the rationals and includes three
Axioms. In this paper, we will keep only the first Axiom. Consequently, by a
\textsl{truncated Riesz space} we shall mean a (non-trivial) Riesz space $E$
along with a \textsl{truncation}, that is, a nonzero map $x\rightarrow
x^{\ast}$ from the positive cone $E^{+}$ into itself such that%
\[
x\wedge y^{\ast}\leq x^{\ast}\leq x,\text{ for all }x,y\in E^{+}.
\]
We may prove with very little effort that a nonzero map $x\rightarrow x^{\ast
}$ is a truncation on $E$ if and only if%
\[
x^{\ast}\wedge y=x\wedge y^{\ast},\text{ for all }x,y\in E^{+}.
\]
However, it might be wondered about the connection between this abstract
definition and the original Fremlin's definition. An answer to this question
was given recently by the first named author and El Adeb \cite{BE2017}.
Indeed, they proved that if $E$ is a truncated Riesz space, then the direct
sum $E\oplus\mathbb{R}$ can be equipped with a non-standard structure of a
Riesz space such that $E$ is a Riesz subspace of $E\oplus\mathbb{R}$ and the
equality%
\[
x^{\ast}=1\wedge x\text{ in }E\oplus\mathbb{R}\text{ for all }x\in E^{+}.
\]
Even though this is not in the agenda, it could be interesting to point out
that the Riesz space $E\oplus\mathbb{R}$, called the \textsl{unitization }of
$E$, is a universal object. For details on universal properties of
$E\oplus\mathbb{R}$, the reader can consult the recent reference
\cite{BHM2018}.

Before discussing the content of the present paper, it would be helpful to
digress a bit and talk about the classical unitization process of a
(non-unital) Banach algebra $A$ \cite{BD1973}. This process is satisfactory
because it produces a Banach algebra $A\oplus\mathbb{R}$ with $1$ as
multiplicative unit, which plays a fundamental role, for instance, in studying
spectral properties of $A$ itself. However, such a technique loses completely
its effectiveness in the context of Banach Lattices. To better understand the
failure, let $E$ be a Banach lattice and assume that the vector space
$E\oplus\mathbb{R}$ is equipped with its usual coordinatewise ordering. Of
course, $E\oplus\mathbb{R}\ $under the natural norm given by%
\[
\left\Vert x+\alpha\right\Vert =\left\Vert x\right\Vert +\left\vert
\alpha\right\vert ,\text{ for all }x\in E\text{ and }\alpha\in\mathbb{R}%
\]
is a Banach lattice. However, the equality $x\wedge1=0\ $holds in
$E\oplus\mathbb{R}\ $for all $x\in E^{+}$. This means that this
lattice-ordered structure on $E\oplus\mathbb{R}$ is superfluous and adds
nothing of substance to $E$.

Our main purpose in the present work is to investigate the unitization
$E\oplus\mathbb{R}$ if the given truncated Riesz space $E$ is simultaneously a
normed Riesz space. To be a little more precise, let $E$ be such a Riesz
space. First of all, we want to know wether or not the lattice norm
$\left\Vert .\right\Vert $ of $E$ can be extended to a lattice norm
$\left\Vert .\right\Vert _{u}$ on $E\oplus\mathbb{R}$. If so, we would have,
for every $x\in E^{+}$,%
\[
\left\Vert x^{\ast}\right\Vert =\left\Vert x^{\ast}\right\Vert _{u}=\left\Vert
1\wedge x\right\Vert _{u}\leq\left\Vert 1\right\Vert _{u}.
\]
Accordingly, a necessary condition for $E\oplus\mathbb{R}$ to be equipped with
a lattice norm that extends $\left\Vert .\right\Vert $ is that the truncation
$x\rightarrow x^{\ast}$ must be norm-bounded, i.e., there exists $M\in\left(
0,\infty\right)  $ such that%
\[
\left\Vert x^{\ast}\right\Vert \leq M,\quad\text{for all }x\in E^{+}.
\]
Surprisingly enough, as we shall see next, it will turn out that this
condition is sufficient as well. In this regard, we call a\textit{ truncated
normed Riesz space} any normed Riesz space $E$ along with a norm-bounded
truncation $x\rightarrow x^{\ast}$. Hence, if $E$ is a truncated normed Riesz
space then the supremum $\sup\left\{  \left\Vert x^{\ast}\right\Vert :x\in
E^{+}\right\}  $ exists. Observe that, by re-norming the truncated normed
Riesz space $E$ in an obvious way, we can assume that this supremum equals $1$.

\begin{quote}
\textsl{Beginning with the next paragraph, we shall impose the equality}%
\[
\sup\left\{  \left\Vert x^{\ast}\right\Vert :x\in E^{+}\right\}  =1
\]
\textsl{as a blanket assumption} \textsl{on any given truncated normed Riesz
space }$E$\textsl{ }(\textsl{unless otherwise stated explicitly})\textsl{.}
\end{quote}

Now, let's give a short synopsis of the results of this paper. Let $E$ be a
truncated normed Riesz space. By a \textsl{unitization norm} on $E\oplus
\mathbb{R}$ is meant a lattice norm $\left\Vert .\right\Vert _{u}$ on
$E\oplus\mathbb{R}$ satisfying the equality $\left\Vert 1\right\Vert _{u}=1$.
We shall prove that there exists a largest unitization norm $\left\Vert
.\right\Vert _{1}$ on $E\oplus\mathbb{R}$. On the other hand, it could happen
that the truncation of $E$ is provided by meet with an element $e>0$ in $E$,
that is,%
\[
x^{\ast}=e\wedge x,\text{ for all }x\in E^{+}.
\]
Such an element is called a \textsl{truncation unit}. It is readily checked
that $E$ has at most one truncation unit. This assumes that $E$ need not have
one. By way of illustration, let $C_{0}\left(  \mathbb{R}\right)  $ denote the
Banach lattice of all continuous real-valued functions on the reals
$\mathbb{R}$. A moment's thought reveals that the formula%
\[
x^{\ast}\left(  t\right)  =\min\left\{  1,x\left(  t\right)  \right\}  ,\text{
for all }x\in C_{0}\left(  \mathbb{R}\right)  \text{ and }t\in\mathbb{R}%
\]
makes $C_{0}\left(  \mathbb{R}\right)  $ into a truncated normed Riesz space
with no truncation unit. It will turn out that if the truncated normed Riesz
space $E$ has no truncation unit, then there exists a smallest unitization
norm $\left\Vert .\right\Vert _{0}$ on $E\oplus\mathbb{R}$. Next comes a
thorough study of arbitrary unitization norms on $E\oplus\mathbb{R}$. In this
prospect, the extreme unitization norms $\left\Vert .\right\Vert _{1}$ and
$\left\Vert .\right\Vert _{0}$ will play the role of reference norms. But we
will first observe that $E$ is a maximal order ideal in $E\oplus\mathbb{R}$
from which it follows that, if $E\oplus\mathbb{R}$ is equipped with a
unitization norm $\left\Vert .\right\Vert _{u}$, then $E$ is either dense or a
closed set in $E\oplus\mathbb{R}$. Then, we shall prove that $E$ is a closed
set in $\left(  E\oplus\mathbb{R},\left\Vert .\right\Vert _{u}\right)  $ if
and only if $\left\Vert .\right\Vert _{u}$ and $\left\Vert .\right\Vert _{1}$
are equivalent. This means in particular that $E$ is a closed set in $\left(
E\oplus\mathbb{R},\left\Vert .\right\Vert _{1}\right)  $. However, examples
are provided to show that we cannot decide how does $E$ sit in $\left(
E\oplus\mathbb{R},\left\Vert .\right\Vert _{0}\right)  $ (of course, here $E$
does not have truncation unit). In spite of that, we will show that if $E$ is
dense in $\left(  E\oplus\mathbb{R},\left\Vert .\right\Vert _{u}\right)  $ for
some unitization norm $\left\Vert .\right\Vert _{u}$ then $E$ has no
truncation unit and $\left\Vert .\right\Vert _{u}=\left\Vert .\right\Vert
_{0}$. As a consequence of all these results, we shall prove that if
$E\oplus\mathbb{R}$ is equipped with any unitization norm, then $E\oplus
\mathbb{R}$ is a Banach lattice if and only if $E$ is a Banach lattice. The
last application of our results is a representation theorem for truncated
Banach lattices, which extends the celebrate Kakutani's representation theorem
\cite[Theorem 2.1.3]{M1991}. Assume that $E$ is a truncated Banach lattice
such that if $x\in E^{+}$ then $\mu x$ is a fixed point of the truncation for
some $\mu\in\left(  0,\infty\right)  $. Hence, the gauge function defined by%
\[
\left\Vert x\right\Vert _{\infty}=\inf\left\{  \lambda\in\left(
0,\infty\right)  :\left(  \frac{1}{\lambda}\left\vert x\right\vert \right)
^{\ast}=\frac{1}{\lambda}\left\vert x\right\vert \right\}  ,\text{ for all
}x\in E
\]
is a lattice norm on $E$ under which $E$ is lattice isometric to $C_{0}\left(
X\right)  $ for some locally compact Hausdorff space.

We point out finally that, by and large, we follow notation and terminology of
the standard monographs \cite{AB2006,S1974} which will be our main references
on Riesz spaces and Banach lattices.

\section{Extreme unitization norms}

Let $E$ be a truncated Riesz space. As pointed out in the introduction,
$E\oplus\mathbb{R}$ is a Riesz space such that $E$ is a Riesz subspace of
$E\oplus\mathbb{R}$ and%
\[
x^{\ast}=1\wedge x,\text{ for all }x\in E^{+}.
\]
Since, by definition, a truncation is not identically zero, we have $1\wedge
x\neq0$ for at least one element $x\in E^{+}$ with $x\neq0$. The absolute
value in $E\oplus\mathbb{R}$ of $x+\alpha\in E\oplus\mathbb{R}$ is given by%
\begin{equation}
\left\vert x+\alpha\right\vert =\left\{
\begin{array}
[c]{l}%
\left\vert x\right\vert -2\left\vert \alpha\right\vert \left(  \frac{1}%
{\alpha}x^{-}\wedge\frac{-1}{\alpha}x^{+}\right)  ^{\ast}+\left\vert
\alpha\right\vert \text{ if }\alpha\neq0\\
\\
\left\vert x\right\vert \text{ if }\alpha=0.
\end{array}
\right.  \label{av}%
\end{equation}
It follows quickly that if $x+\alpha\geq0$ in $E\oplus\mathbb{R}$ then
$\alpha\geq0$ in $\mathbb{R}$ and%
\begin{equation}
\left\vert x+\alpha\right\vert -\left\vert \alpha\right\vert \in E,\text{ for
all }x\in E\text{ and }\alpha\in\mathbb{R}. \label{in}%
\end{equation}
Proofs of these properties can be found in the recent papers
\cite{BE2017,BHM2018}. On the other hand, the set%
\[
E_{\ast}=\left\{  x\in E:\left\vert x\right\vert ^{\ast}=\left\vert
x\right\vert \right\}
\]
will prove very useful to our work. For instance, we have%
\[
E_{\ast}=\left\{  x\in E:\left\vert x\right\vert \leq1\text{ in }%
E\oplus\mathbb{R}\right\}
\]
and%
\[
1=\sup\left\{  \left\Vert x^{\ast}\right\Vert :x\in E^{+}\right\}
=\sup\left\{  \left\Vert x\right\Vert :x\in E_{\ast}\right\}  .
\]
Moreover, if $E$ is a truncated normed Riesz space and $\left\vert
x\right\vert \leq\alpha$ in $E\oplus\mathbb{R}$ for some $x\in E$ and
$\alpha\in\left(  0,\infty\right)  $ then $\left\Vert x\right\Vert \leq\alpha
$. Indeed, if $y\in E$ then $\ \left\vert y\right\vert \leq1$ in
$E\oplus\mathbb{R}$. Thus, from $\left\vert x\right\vert \leq\alpha$ it
follows that $\frac{1}{\alpha}x\in E_{\ast}$ and so%
\[
\frac{1}{\alpha}\left\Vert x\right\Vert \leq\sup\left\{  \left\Vert
y\right\Vert :x\in E_{\ast}\right\}  =1.
\]
This simple fact will be used below without further mention. The first main
result of this paper is now in order.

\begin{theorem}
\label{largest}Let $E$ be a truncated normed Riesz space. The formula%
\[
\left\Vert x+\alpha\right\Vert _{1}=\left\Vert \left(  \left\vert
x+\alpha\right\vert -\left\vert \alpha\right\vert \right)  ^{+}\right\Vert
+\left\vert \alpha\right\vert ,\text{\quad for all }x\in E\text{ and }%
\alpha\in\mathbb{R}%
\]
defines the largest unitization norm on $E\oplus\mathbb{R}$.
\end{theorem}

\begin{proof}
First, notice that%
\[
\left\Vert x\right\Vert _{1}=\left\Vert x\right\Vert ,\text{ for all }x\in E.
\]
Now, to prove that $\left\Vert .\right\Vert _{1}$ is a norm on $E\oplus
\mathbb{R}$, the only point that needs details is the triangle inequality. For
brevity, denote%
\[
\left\langle x,\alpha\right\rangle =\left\vert x+\alpha\right\vert -\left\vert
\alpha\right\vert ,\text{\quad for all }x\in E\text{ and }\alpha\in
\mathbb{R}.
\]
By (\ref{in}), we have%
\[
\left\langle x,\alpha\right\rangle \in E,\text{\quad for all }x\in E\text{ and
}\alpha\in\mathbb{R}.
\]
Let $x,y\in E$ and $\alpha,\beta\in\mathbb{R}$. We have to show that%
\begin{equation}
\left\Vert x+y+\alpha+\beta\right\Vert _{1}\leq\left\Vert x+\alpha\right\Vert
_{1}+\left\Vert y+\beta\right\Vert _{1}. \label{trian}%
\end{equation}
First, assume that $\alpha\beta\geq0$, which means that $\left\vert
\alpha+\beta\right\vert =\left\vert \alpha\right\vert +\left\vert
\beta\right\vert $. Hence,%
\[
\left\langle x+y,\alpha+\beta\right\rangle \leq\left\vert x+\alpha\right\vert
+\left\vert y+\beta\right\vert -\left\vert \alpha\right\vert -\left\vert
\beta\right\vert ,
\]
so,%
\[
\left\langle x+y,\alpha+\beta\right\rangle ^{+}\leq\left\langle x,\alpha
\right\rangle ^{+}+\left\langle y,\beta\right\rangle ^{+}.
\]
This inequality takes place in $E$ and thus%
\[
\left\Vert \left\langle x+y,\alpha+\beta\right\rangle ^{+}\right\Vert
\leq\left\Vert \left\langle x,\alpha\right\rangle ^{+}\right\Vert +\left\Vert
\left\langle y,\beta\right\rangle ^{+}\right\Vert .
\]
This leads quickly to the inequality (\ref{trian}). Now, assume that
$\alpha\beta<0$, that is, $\left\vert \alpha+\beta\right\vert <\left\vert
\alpha\right\vert +\left\vert \beta\right\vert $. From%
\[
\left\vert x+y+\alpha+\beta\right\vert \leq\left\vert x+\alpha\right\vert
+\left\vert y+\beta\right\vert
\]
it follows that%
\[
\left\langle x+y,\alpha+\beta\right\rangle -\left\langle x,\alpha\right\rangle
-\left\langle y,\beta\right\rangle \leq\left\vert \alpha\right\vert
+\left\vert \beta\right\vert -\left\vert \alpha+\beta\right\vert .
\]
But then%
\[
\left(  \left\langle x+y,\alpha+\beta\right\rangle -\left\langle
x,\alpha\right\rangle -\left\langle y,\beta\right\rangle \right)  ^{+}%
\leq\left\vert \alpha\right\vert +\left\vert \beta\right\vert -\left\vert
\alpha+\beta\right\vert .
\]
This yields that%
\[
\left\Vert \left(  \left\langle x+y,\alpha+\beta\right\rangle -\left\langle
x,\alpha\right\rangle -\left\langle y,\beta\right\rangle \right)
^{+}\right\Vert \leq\left\vert \alpha\right\vert +\left\vert \beta\right\vert
-\left\vert \alpha+\beta\right\vert .
\]
Moreover,%
\[
\left\langle x+y,\alpha+\beta\right\rangle ^{+}\leq\left(  \left\langle
x+y,\alpha+\beta\right\rangle -\left\langle x,\alpha\right\rangle
-\left\langle y,\beta\right\rangle \right)  ^{+}+\left\langle x,\alpha
\right\rangle ^{+}+\left\langle y,\beta\right\rangle ^{+}.
\]
Therefore,%
\[
\left\Vert \left\langle x+y,\alpha+\beta\right\rangle ^{+}\right\Vert
\leq\left\Vert \left\langle x,\alpha\right\rangle ^{+}\right\Vert +\left\vert
\alpha\right\vert +\left\Vert \left\langle y,\beta\right\rangle ^{+}%
\right\Vert +\left\vert \beta\right\vert -\left\vert \alpha+\beta\right\vert
\]
and the inequality (\ref{trian}) follows.

At this point, we prove that $\left\Vert .\right\Vert $ is a lattice norm.
Pick $x,y\in E$ and $\alpha,\beta\in\mathbb{R}$ such that $\left\vert
x+\alpha\right\vert \leq\left\vert y+\beta\right\vert $. In particular, we
have $\left\vert \alpha\right\vert \leq\left\vert \beta\right\vert $. If
$\left\vert \alpha\right\vert =\left\vert \beta\right\vert $ then
$\left\langle x,\alpha\right\rangle \leq\left\langle y,\beta\right\rangle $
and so $\left\langle x,\alpha\right\rangle ^{+}\leq\left\langle y,\beta
\right\rangle ^{+}$. We get $\left\Vert \left\langle x,\alpha\right\rangle
^{+}\right\Vert \leq\left\Vert \left\langle y,\beta\right\rangle
^{+}\right\Vert $ from which it follows that $\left\Vert x+\alpha\right\Vert
_{1}\leq\left\Vert y+\beta\right\Vert _{1}$. Now, suppose $\left\vert
\alpha\right\vert <\left\vert \beta\right\vert $. Thus,%
\[
\left\langle x,\alpha\right\rangle -\left\langle y,\beta\right\rangle
\leq\left\vert \beta\right\vert -\left\vert \alpha\right\vert
\]
and so%
\[
\left(  \left\langle x,\alpha\right\rangle -\left\langle y,\beta\right\rangle
\right)  ^{+}\leq\left\vert \beta\right\vert -\left\vert \alpha\right\vert .
\]
We get%
\[
\left\Vert \left(  \left\langle x,\alpha\right\rangle -\left\langle
y,\beta\right\rangle \right)  ^{+}\right\Vert \leq\left\vert \beta\right\vert
-\left\vert \alpha\right\vert .
\]
Furthermore,%
\[
\left\langle x,\alpha\right\rangle ^{+}\leq\left(  \left\langle x,\alpha
\right\rangle -\left\langle y,\beta\right\rangle \right)  ^{+}+\left\langle
y,\beta\right\rangle ^{+}.
\]
Accordingly,%
\[
\left\Vert \left\langle x,\alpha\right\rangle ^{+}\right\Vert \leq\left\Vert
\left(  \left\langle x,\alpha\right\rangle -\left\langle y,\beta\right\rangle
\right)  ^{+}\right\Vert +\left\Vert \left\langle y,\beta\right\rangle
^{+}\right\Vert \leq\left\vert \beta\right\vert -\left\vert \alpha\right\vert
+\left\Vert \left\langle y,\beta\right\rangle ^{+}\right\Vert .
\]
We derive that $\left\Vert x+\alpha\right\Vert _{1}\leq\left\Vert
y+\beta\right\Vert _{1}$, meaning that $\left\Vert .\right\Vert _{1}$ is a
lattice norm on $E\oplus\mathbb{R}$. Observe now that $\left\Vert 1\right\Vert
_{1}=1$ and so $\left\Vert .\right\Vert _{1}$ is a unitization norm on
$E\oplus\mathbb{R}$.

Finally, we claim that $\left\Vert .\right\Vert _{1}$ is the largest
unitization norm on $E\oplus\mathbb{R}$. To this end, assume that $\left\Vert
.\right\Vert _{u}$ is another unitization norm of $E\oplus\mathbb{R}$ and pick
$x+\alpha\in E\oplus\mathbb{R}$. If $\alpha=0$ then%
\[
\left\Vert x\right\Vert _{u}=\left\Vert x\right\Vert =\left\Vert x\right\Vert
_{1}.
\]
Suppose that $\alpha\neq0$ and observe that, by (\ref{in}),%
\begin{align*}
\left\Vert x+\alpha\right\Vert _{u}  &  =\left\Vert \left\vert x+\alpha
\right\vert \right\Vert _{u}\\
&  =\left\Vert \left\vert x+\alpha\right\vert -\left\vert \alpha\right\vert
+\left\vert \alpha\right\vert \right\Vert _{u}\\
&  =\left\Vert \left(  \left\vert x+\alpha\right\vert -\left\vert
\alpha\right\vert \right)  ^{+}-\left(  \left\vert x+\alpha\right\vert
-\left\vert \alpha\right\vert \right)  ^{-}+\left\vert \alpha\right\vert
\right\Vert _{u}\\
&  \leq\left\Vert \left(  \left\vert x+\alpha\right\vert -\left\vert
\alpha\right\vert \right)  ^{+}\right\Vert _{u}+\left\Vert \left(  \left\vert
x+\alpha\right\vert -\left\vert \alpha\right\vert \right)  ^{-}-\left\vert
\alpha\right\vert \right\Vert _{u}\\
&  =\left\Vert \left(  \left\vert x+\alpha\right\vert -\left\vert
\alpha\right\vert \right)  ^{+}\right\Vert +\left\Vert \left(  \left\vert
x+\alpha\right\vert -\left\vert \alpha\right\vert \right)  ^{-}-\left\vert
\alpha\right\vert \right\Vert _{u}.
\end{align*}
On the other hand,%
\[
0\leq\left\vert x+\alpha\right\vert \wedge\left\vert \alpha\right\vert
=\left\vert \alpha\right\vert -\left(  \left\vert x+\alpha\right\vert
-\left\vert \alpha\right\vert \right)  ^{-}\leq\left\vert \alpha\right\vert .
\]
But then%
\[
\left\Vert \left(  \left\vert x+\alpha\right\vert -\left\vert \alpha
\right\vert \right)  ^{-}-\left\vert \alpha\right\vert \right\Vert _{u}%
\leq\left\Vert \left\vert \alpha\right\vert \right\Vert _{u}=\left\vert
\alpha\right\vert .
\]
We derive that%
\[
\left\Vert x+\alpha\right\Vert _{u}\leq\left\Vert \left(  \left\vert
x+\alpha\right\vert -\left\vert \alpha\right\vert \right)  ^{+}\right\Vert
+\left\vert \alpha\right\vert =\left\Vert x+\alpha\right\Vert _{1}.
\]
This completes the proof of the theorem.
\end{proof}

Recall that an element $e>0$ in a truncated Riesz space is called a
\textsl{truncation unit} if the truncation is provided by meet with $e$,
meaning that $x^{\ast}=e\wedge x$ for all $x\in E^{+}$. The following result
will be useful for later purposes.

\begin{proposition}
\label{ideal}Let $E$ be a truncated Riesz space. Then $E$ is maximal order
ideal in $E\oplus\mathbb{R}$.
\end{proposition}

\begin{proof}
Pick $x,y\in E$ and $\alpha\in\mathbb{R}$ such that $\left\vert x+\alpha
\right\vert \leq\left\vert y\right\vert $. We claim that $\alpha=0$.
Otherwise, we would have, by (\ref{av}),%
\[
0\leq\left\vert y\right\vert -\left\vert x+\alpha\right\vert =\left[
\left\vert y\right\vert -\left\vert x\right\vert +2\left\vert \alpha
\right\vert \left(  \frac{1}{\alpha}x^{-}\wedge\frac{-1}{\alpha}x^{+}\right)
^{\ast}\right]  -\left\vert \alpha\right\vert .
\]
But then $-\left\vert \alpha\right\vert $ should be positive and so $\alpha
=0$. This contradiction yields that $E$ is an order ideal in $E\oplus
\mathbb{R}$. Moreover, $1\notin E$ and so $E$ is a maximal order ideal in
$E\oplus\mathbb{R}$, completing the proof of the proposition.
\end{proof}

The second main theorem of our work follows next.

\begin{theorem}
\label{smallest}Let $E$ be a truncated normed Riesz space with no truncation
unit. Then the function that takes each $x+\alpha\in E\oplus\mathbb{R}$ to the
positive real number%
\[
\left\Vert x+\alpha\right\Vert _{0}=\sup\left\{  \left\Vert y\right\Vert :y\in
E\text{ and }\left\vert y\right\vert \leq\left\vert x+\alpha\right\vert
\right\}
\]
is the smallest unitization norm on $E\oplus\mathbb{R}$.
\end{theorem}

\begin{proof}
Let $x+\alpha\in E\oplus\mathbb{R}$ and put%
\[
A=\left\{  \left\Vert y\right\Vert :y\in E\text{ and }\left\vert y\right\vert
\leq\left\vert x+\alpha\right\vert \right\}  .
\]
Clearly, $0\in A\neq\emptyset$ and if $y\in E$ with $\left\vert y\right\vert
\leq\left\vert x+\alpha\right\vert $ then%
\[
\left\Vert y\right\Vert =\left\Vert y\right\Vert _{1}\leq\left\Vert
x+\alpha\right\Vert _{1}.
\]
It follows that $A$ has a supremum $\left\Vert x+\alpha\right\Vert _{0}$.
Moreover, it is trivial that $\left\Vert x+\alpha\right\Vert _{0}=\left\Vert
\left\vert x+\alpha\right\vert \right\Vert _{0}$. Also, it is evident that
$\left\Vert x\right\Vert _{0}=\left\Vert x\right\Vert $. Now, assume that
$0\leq x+\alpha$ and notice that, in this situation, $\alpha\geq0$. We claim
that $x+\alpha=0$ if $\left\Vert x+\alpha\right\Vert _{0}=0$. We have nothing
to establish if $\alpha=0$. So, we suppose that $\alpha>0$. From $x+\alpha
\geq0$ it follows directly that $x^{-}\leq\alpha$. Thus,%
\[
x^{+}=x+x^{-}\leq x+\alpha=\left\vert x+\alpha\right\vert
\]
and so $\left\Vert x^{+}\right\Vert \in A$. Therefore, $\left\Vert
x^{+}\right\Vert \leq\left\Vert x+\alpha\right\Vert _{0}=0$. Hence, $x^{+}=0$.
Put $w=\alpha^{-1}x^{-}$ and pick $y\in E^{+}$. Clearly, $w\leq1$ from which
we derive that%
\[
0\leq-x^{-}+\alpha\left(  1\wedge\left(  w+y\right)  \right)  \leq
-x^{-}+\alpha=x+\alpha.
\]
Then,%
\[
\left\Vert -x^{-}+\alpha\left(  1\wedge\left(  w+y\right)  \right)
\right\Vert \leq\left\Vert x+\alpha\right\Vert _{0}=0.
\]
Accordingly,%
\[
1\wedge\left(  w+y\right)  =\alpha^{-1}x^{-}=w.
\]
Whence,%
\[
1\wedge y=1\wedge y\wedge\left(  w+y\right)  =y\wedge w.
\]
This means that $w$ is a truncation unit in $E$ which contradicts the
hypothesis. We get $\alpha=0$ and so $x+\alpha=0$ (since $x^{+}=0$ and
$x^{-}\leq\alpha$). Now, let $r$ be a nonzero real number and $x+\alpha
,y+\beta\in E\oplus\mathbb{R}$. Then%
\begin{align*}
\left\Vert r\left(  x+\alpha\right)  \right\Vert _{0}  &  =\sup\left\{
\left\Vert z\right\Vert :z\in E\text{ and }\left\vert z\right\vert
\leq\left\vert r\left(  x+\alpha\right)  \right\vert \right\} \\
&  =\sup\left\{  \left\Vert z\right\Vert :z\in E\text{ and }\left\vert
r^{-1}z\right\vert \leq\left\vert x+\alpha\right\vert \right\} \\
&  =\left\vert r\right\vert \sup\left\{  \left\Vert r^{-1}z\right\Vert :z\in
E\text{ and }\left\vert r^{-1}z\right\vert \leq\left\vert x+\alpha\right\vert
\right\} \\
&  =\left\vert r\right\vert \sup\left\{  \left\Vert z\right\Vert :z\in E\text{
and }\left\vert z\right\vert \leq\left\vert x+\alpha\right\vert \right\}
=\left\vert r\right\vert \left\Vert x+\alpha\right\Vert _{0}.
\end{align*}
On the other hand, if $z\in E$ and $\left\vert z\right\vert \leq\left\vert
x+y+\alpha+\beta\right\vert $ then%
\begin{align*}
\left\vert z\right\vert  &  =\left\vert z\right\vert \wedge\left\vert
x+y+\alpha+\beta\right\vert \\
&  \leq\left\vert z\right\vert \wedge\left(  \left\vert x+\alpha\right\vert
+\left\vert y+\beta\right\vert \right) \\
&  \leq\left(  \left\vert z\right\vert \wedge\left\vert x+\alpha\right\vert
\right)  +\left(  \left\vert z\right\vert \wedge\left\vert y+\beta\right\vert
\right)  .
\end{align*}
Since $E$ is an order ideal in $E\oplus\mathbb{R}$ (see Proposition
\ref{ideal}), we have $\left\vert z\right\vert \wedge\left\vert x+\alpha
\right\vert \in E$ and $\left\vert z\right\vert \wedge\left\vert
y+\beta\right\vert \in E$. Accordingly,%

\begin{align*}
\left\Vert z\right\Vert  &  \leq\left\Vert \left(  \left\vert z\right\vert
\wedge\left\vert x+\alpha\right\vert \right)  +\left(  \left\vert z\right\vert
\wedge\left\vert y+\beta\right\vert \right)  \right\Vert \\
&  \leq\left\Vert \left(  \left\vert z\right\vert \wedge\left\vert
x+\alpha\right\vert \right)  \right\Vert +\left\Vert \left(  \left\vert
z\right\vert \wedge\left\vert y+\beta\right\vert \right)  \right\Vert \\
&  \leq\left\Vert x+\alpha\right\Vert _{0}+\left\Vert y+\beta\right\Vert _{0}.
\end{align*}
Hence,%
\[
\left\Vert x+\alpha+y+\beta\right\Vert _{0}\leq\left\Vert x+\alpha\right\Vert
_{0}+\left\Vert y+\beta\right\Vert _{0}%
\]
and finally $\left\Vert .\right\Vert _{0}$ is a lattice norm on $E\oplus
\mathbb{R}$. It remains to show that $\left\Vert 1\right\Vert _{0}=1$. Indeed,%
\[
\left\Vert 1\right\Vert _{0}=\sup\left\{  \left\Vert x\right\Vert :x\in
E\text{ and }\left\vert x\right\vert \leq1\right\}  =\sup\left\{  \left\Vert
x\right\Vert :x\in E_{\ast}\right\}  .
\]
We have%
\[
1=\inf\left\{  M>0:\left\Vert x\right\Vert \leq M\text{ for all }x\in E_{\ast
}\right\}  .
\]
This implies that $1\leq\left\Vert 1\right\Vert _{0}$ because $\left\Vert
1\right\Vert _{0}\geq\left\Vert x\right\Vert $ for all $x\in E_{\ast}$.
Conversely, it is easily seen that $\left\Vert x\right\Vert \leq1$ for all
$x\in E_{\ast}$. So $\left\Vert 1\right\Vert _{0}\leq1$, meaning that
$\left\Vert 1\right\Vert _{0}=1$. Let $\left\Vert .\right\Vert _{u}$ be a
unitization norm on $E$. If $x+\alpha\in E\oplus\mathbb{R}$ and $y\in E$ with
$\left\vert y\right\vert \leq\left\vert x+\alpha\right\vert $ then%
\[
\left\Vert y\right\Vert =\left\Vert y\right\Vert _{u}\leq\left\Vert
x+\alpha\right\Vert _{u}.
\]
It follows that $\left\Vert x+\alpha\right\Vert _{u}$ is an upper bound of
$A$. Thus
\[
\left\Vert x+\alpha\right\Vert _{0}=\sup A\leq\left\Vert x+\alpha\right\Vert
_{u},
\]
completing the proof of the theorem.
\end{proof}

Let's say a few words on $\left\Vert .\right\Vert _{0}$ when $E$ has a
unitization unit. To this end, we need the following lemma, which will be
useful for later purpose also.

\begin{lemma}
\label{unit}Let $E$ be a truncated Riesz space with a truncation unit $e$.
Then $0<e<1$ in $E\oplus\mathbb{R}$ and%
\[
\left\vert x\right\vert \wedge\left(  1-e\right)  =0\text{ in }E\oplus
\mathbb{R}\text{ for all }x\in E.
\]

\end{lemma}

\begin{proof}
As $1\notin E$ and $e=e\wedge e=e^{\ast}=e\wedge1,$ we get $0<e\leq1$. Pick
$x\in E$ and write%
\begin{align*}
\left\vert x\right\vert \wedge\left(  1-e\right)   &  =\left(  \left(
e+\left\vert x\right\vert \right)  \wedge1\right)  -e=\left(  e+\left\vert
x\right\vert \right)  ^{\ast}-e\\
&  =\left(  e\wedge\left(  e+\left\vert x\right\vert \right)  \right)
-e=e-e=0.
\end{align*}
This leads to the desired result.
\end{proof}

So, let $E$ be a truncated normed Riesz space with a truncation unit $e$. Pick
$x+\alpha\in E\oplus\mathbb{R}$ and observe that if $y\in E$ and $\left\vert
y\right\vert \leq\left\vert x+\alpha\right\vert $ then%
\[
\left\vert y\right\vert \leq\left\vert x+\alpha\right\vert =\left\vert
x+\alpha e+\alpha\left(  1-e\right)  \right\vert .
\]
By Lemma \ref{unit}, we get $\left\vert y\right\vert \leq\left\vert x+\alpha
e\right\vert $. Thus,%
\[
\left\Vert x+\alpha\right\Vert _{0}=\sup\left\{  \left\Vert y\right\Vert :y\in
E\text{ and }\left\vert y\right\vert \leq\left\vert x+\alpha\right\vert
\right\}  =\left\Vert x+\alpha e\right\Vert .
\]
Taking $x=e$ and $\alpha=-1$, we see that $e-1\neq0$ thought $\left\Vert
e-1\right\Vert _{0}=0$, meaning that $\left\Vert .\right\Vert _{0}$ is not
even a norm. In spite of that, reading carefully the proof of Theorem
\ref{smallest}, we see that $\left\Vert .\right\Vert _{0}$ is always a lattice
semi-norm on $E\oplus\mathbb{R}$.

\section{Arbitrary unitization norms}

Assume that $E$ is a truncated normed Riesz space and that $E\oplus\mathbb{R}$
is equipped with any unitization norm $\left\Vert .\right\Vert _{u}$. By
Proposition \ref{ideal}, $E$ is an order ideal in $E\oplus\mathbb{R}$ and so
is its closure in $E\oplus\mathbb{R}$ (see, e.g., Theorem 15.19 in
\cite{Z1997}). We infer straightforwardly, by maximality, that $E$ is either
dense or a closed set in $E\oplus\mathbb{R}$. This observation will come in handy.

\begin{theorem}
\label{closed}Let $E$ be a truncated normed Riesz space and $E\oplus
\mathbb{R}$ be equipped with a unitization norm $\left\Vert .\right\Vert _{u}%
$. Then $E$ is a closed set in $E\oplus\mathbb{R}$ if and only if $\left\Vert
.\right\Vert _{u}$ and $\left\Vert .\right\Vert _{1}$ are equivalent.
\end{theorem}

\begin{proof}
Suppose that $E$ is a closed set in $E\oplus\mathbb{R}$. Clearly,%
\[
0<r=\inf\left\{  \left\Vert x-1\right\Vert _{u}:x\in E_{\ast}\right\}  .
\]
Pick $x\in E_{\ast}$ and observe that $\left\vert x\right\vert \leq1$. We
derive that $0\leq-x+1\leq2$ and so $\left\Vert x-1\right\Vert _{u}\leq2$.
This yields that $0<r\leq2$. Now, let $x\in E$ and $\alpha\in\mathbb{R}$ such
that $x+\alpha\geq0$. In particular, $\alpha\geq0$. We claim that%
\[
r\left\Vert x+\alpha\right\Vert _{1}\leq3\left\Vert x+\alpha\right\Vert _{u}.
\]
The inequality being clear for $\alpha=0$, suppose that $\alpha>0$. We have%
\[
\left(  \left\vert x+\alpha\right\vert -\alpha\right)  ^{+}=x^{+}\vee\left(
-x-2\alpha\right)  =x^{+}.
\]
Consequently,%
\[
\left\Vert x+\alpha\right\Vert _{1}=\left\Vert \left(  \left\vert
x+\alpha\right\vert -\alpha\right)  ^{+}\right\Vert +\alpha=\left\Vert
x^{+}\right\Vert +\alpha.
\]
Moreover, $\frac{1}{\alpha}x^{-}\in E^{\ast}$ and so $\left\Vert \frac
{1}{\alpha}x^{-}-1\right\Vert _{u}\geq r$. We get $\left\Vert x^{-}%
-\alpha\right\Vert _{u}\geq\alpha r$ from which it follows that%
\[
r\left\Vert x+\alpha\right\Vert _{1}=r\left\Vert x^{+}\right\Vert +\alpha
r\leq r\left\Vert x^{+}\right\Vert +\left\Vert x^{-}-\alpha\right\Vert
_{u}\leq2\left\Vert x^{+}\right\Vert +\left\Vert x^{-}-\alpha\right\Vert
_{u}.
\]
Since $x^{-}-\alpha\leq0$, we obtain%
\[
r\left\Vert x+\alpha\right\Vert _{1}\leq2\left\Vert x+\alpha\right\Vert
_{u}+\left\Vert x+\alpha\right\Vert _{u}=3\left\Vert x+\alpha\right\Vert
_{u},
\]
as desired. In view of Theorem \ref{largest}, we conclude that $\left\Vert
.\right\Vert _{u}$ and $\left\Vert .\right\Vert _{1}$ are equivalent.

Conversely, it suffices to show that $E$ is a closed set in $\left(
E\oplus\mathbb{R},\left\Vert .\right\Vert _{1}\right)  $. Otherwise, $E$ would
be dense in $E\oplus\mathbb{R}$. In particular, it would exist a sequence
$\left(  x_{n}\right)  $ in $E$ such that%
\[
\lim x_{n}=1\text{ in }\left(  E\oplus\mathbb{R},\left\Vert .\right\Vert
_{1}\right)  .
\]
But this is impossible since if $x\in E$ then%
\[
\left\Vert x-1\right\Vert _{1}=\left\Vert \left(  \left\vert x-1\right\vert
-1\right)  ^{+}\right\Vert +1\geq1.
\]
This completes the proof of the theorem.
\end{proof}

We have seen in Theorem \ref{closed} that any truncated normed Riesz space $E$
is a closed set in $\left(  E\oplus\mathbb{R},\left\Vert .\right\Vert
_{1}\right)  $. Nevertheless, the picture is much less clear for the norm
$\left\Vert .\right\Vert _{0}$ as the following examples illustrate.

\begin{example}
As usual, let $C\left(  X\right)  $ denote the Riesz space of all real-valued
continuous functions on a topological space $X$.

\begin{enumerate}
\item[\emph{(i)}] Clearly, the set%
\[
E=\left\{  f\in C\left(  \left[  -1,1\right]  \right)  :f\left(  -1\right)
=f\left(  1\right)  =0\right\}
\]
is a Riesz subspace of $C\left(  \left[  -1,1\right]  \right)  $. Moreover,
the formula%
\[
\left\Vert f\right\Vert =\frac{1}{2}\int_{-1}^{1}\left\vert f\left(  s\right)
\right\vert ds,\text{ for all }f\in E
\]
defines a lattice norm on $E$. On the other hand, observe that the equality%
\[
f^{\ast}\left(  r\right)  =\min\left\{  1,f\left(  r\right)  \right\}  ,\text{
for all }f\in E\text{ and }r\in\left[  -1,1\right]
\]
makes $E$ into a truncated normed Riesz space with no truncation unit.
Consider now the unitization norm on $E\oplus\mathbb{R}$ defined by%
\[
\left\Vert f+\alpha\right\Vert _{u}=\frac{1}{2}\int_{-1}^{1}\left\vert
f\left(  s\right)  +\alpha\right\vert ds,\text{ for all }f\in E
\]
Furthermore, if $n\in\left\{  2,3,...\right\}  $ we define $f_{n}\in E$ by
putting%
\[
f_{n}\left(  r\right)  =\left\{
\begin{array}
[c]{l}%
n\left(  r+1\right)  \text{ if }r\in\left[  -1,-1+\frac{1}{n}\right] \\
\\
1\text{ if }r\in\left(  -1+\frac{1}{n},1-\frac{1}{n}\right) \\
\\
n\left(  1-r\right)  \text{ if }r\in\left[  1-\frac{1}{n},1\right]  .
\end{array}
\right.
\]
A simple calculation yields that%
\[
\left\Vert f_{n}-1\right\Vert _{u}=\frac{1}{n}\text{, for all }n\in\left\{
2,3,...\right\}  .
\]
This together with \emph{Theorem \ref{smallest}} shows that $1$ belongs to the
closure of $E$ in $\left(  E\oplus\mathbb{R},\left\Vert .\right\Vert
_{0}\right)  $. Accordingly, $E$ is not closed in $\left(  E\oplus
\mathbb{R},\left\Vert .\right\Vert _{0}\right)  $, so $E$ must be dense in
$\left(  E\oplus\mathbb{R},\left\Vert .\right\Vert _{0}\right)  $.

\item[\emph{(ii)}] The set $E$ of all functions in $C\left(  \mathbb{R}%
\right)  $ with compact support is a Riesz subspace of $C\left(
\mathbb{R}\right)  $. The uniform norm defined on $E$ by%
\[
\left\Vert f\right\Vert =\sup\left\{  \left\vert f\left(  r\right)
\right\vert :r\in\mathbb{R}\right\}  ,\text{ for all }f\in E
\]
is a lattice norm on $E$. It is an easy exercise to check that $E$ is a
truncated normed Riesz space under the truncation given by%
\[
f^{\ast}\left(  r\right)  =\min\left\{  1,f\left(  r\right)  \right\}  ,\text{
for all }f\in E\text{ and }r\in\mathbb{R}.
\]
Also, we can easily show that $E$ has no truncation unit. Now, let $f\in E$
such that%
\[
0\leq f\left(  r\right)  \leq1,\text{ for all }r\in\mathbb{R}%
\]
and pick $a\in\left(  0,\infty\right)  $ such that the support of $f$ is
included in the real interval $\left[  -a,a\right]  $. Define $g\in E$ by
putting%
\[
g\left(  r\right)  =\left\{
\begin{array}
[c]{l}%
f\left(  r\right)  \text{ if }r\in\left(  -\infty,a\right) \\
\\
r-a\text{ if }r\in\left[  a,a+1\right) \\
\\
a+2-r\text{ if }r\in\left[  a+1,a+2\right) \\
\\
0\text{ if }r\in\left[  a+2,\infty\right)  .
\end{array}
\right.
\]
Clearly, we have%
\[
0\leq f\left(  r\right)  \leq g\left(  r\right)  \leq1,\text{ for all }%
r\in\mathbb{R}.
\]
It follows that%
\[
\left\Vert -f+1\right\Vert _{0}=\sup\left\{  \left\Vert h\right\Vert
:\left\vert h\right\vert \leq-f+1\right\}  \geq\left\Vert g\right\Vert =1.
\]
We derive that the closure of $E$ in $\left(  E\oplus\mathbb{R},\left\Vert
.\right\Vert _{0}\right)  $ does not contain $1$. In particular, $E$ is not
dense in $\left(  E\oplus\mathbb{R},\left\Vert .\right\Vert _{0}\right)  $,
meaning that $E\ $is a closed set in $\left(  E\oplus\mathbb{R},\left\Vert
.\right\Vert _{0}\right)  $.
\end{enumerate}
\end{example}

Hence, unlike the $\left\Vert .\right\Vert _{1}$ case, we cannot decide wether
$E$ is dense or a closed set in $\left(  E\oplus\mathbb{R},\left\Vert
.\right\Vert _{0}\right)  $. In spite of that, we have the following
\textquotedblleft dense\textquotedblright\ version of Theorem \ref{closed},
which is the last result of this section.

\begin{theorem}
\label{dense}Let $E$ be a truncated normed Riesz space and $\left\Vert
.\right\Vert _{u}$ be a unitization norm on $E\oplus\mathbb{R}$. If $E$ is
dense in $\left(  E\oplus\mathbb{R},\left\Vert .\right\Vert _{u}\right)  $
then $E$ has no truncation unit and $\left\Vert .\right\Vert _{u}=\left\Vert
.\right\Vert _{0}$.
\end{theorem}

\begin{proof}
Suppose that $E$ is dense in $\left(  E\oplus\mathbb{R},\left\Vert
.\right\Vert _{u}\right)  $. Arguing by contradiction, we assume that $E$ has
a truncation unit $e$. By density, there exists a sequence $\left(
x_{n}\right)  $ in $E$ such that%
\[
\lim\left\vert x_{n}-e\right\vert =1-e\text{ in }\left(  E\oplus
\mathbb{R},\left\Vert .\right\Vert _{u}\right)  .
\]
However, using Lemma \ref{unit}, we can write%
\[
\left\vert x_{n}-e\right\vert \wedge\left(  1-e\right)  =0,\text{ for all
}n\in\left\{  1,2,...\right\}  .
\]
This leads directly to the contraction $1=e\in E$ and proves that $E$ does not
have a truncation unit. Now, we claim that $\left\Vert .\right\Vert
_{u}=\left\Vert .\right\Vert _{0}$. Let $x\in E$ and $\alpha\in\mathbb{R}$.
Since $E$ is dense in $\left(  E\oplus\mathbb{R},\left\Vert .\right\Vert
_{u}\right)  $, there exists a sequence $\left(  x_{n}\right)  $ in $E$ such
that%
\[
\lim x_{n}=x+\alpha\text{ in }\left(  E\oplus\mathbb{R},\left\Vert
.\right\Vert _{u}\right)  .
\]
Thus, by Theorem \ref{smallest}, if $n\in\left\{  1,2,...\right\}  $ then%
\begin{align*}
\left\vert \left\Vert x+\alpha\right\Vert _{0}-\left\Vert x+\alpha\right\Vert
_{u}\right\vert  &  \leq\left\vert \left\Vert x+\alpha\right\Vert
_{0}-\left\Vert x_{n}\right\Vert \right\vert +\left\vert \left\Vert
x_{n}\right\Vert -\left\Vert x+\alpha\right\Vert _{u}\right\vert \\
&  \leq\left\Vert x-x_{n}+\alpha\right\Vert _{0}+\left\Vert x-x_{n}%
+\alpha\right\Vert _{u}\\
&  \leq2\left\Vert x-x_{n}+\alpha\right\Vert _{u}.
\end{align*}
It follows quickly that $\left\Vert x+\alpha\right\Vert _{0}-\left\Vert
x+\alpha\right\Vert _{u}=0$ and the theorem follows.
\end{proof}

\section{Unitization of truncated Banach lattice}

We start our investigation by the following technical lemma.

\begin{lemma}
\label{classic}Let $E$ be a truncated normed Riesz space and $x+\alpha\in
E\oplus\mathbb{R}$. Then%
\[
\frac{\left\Vert x\right\Vert +\left\vert \alpha\right\vert }{3}\leq\left\Vert
x+\alpha\right\Vert _{1}\leq3\left(  \left\Vert x\right\Vert +\left\vert
\alpha\right\vert \right)  .
\]

\end{lemma}

\begin{proof}
The inequalities are obvious if $\alpha=0$. So, suppose that $\alpha\neq0$.
Hence, by (\ref{av}), we have%
\[
\left\Vert x+\alpha\right\Vert _{1}=\left\Vert \left\vert x+\alpha\right\vert
\right\Vert _{1}=\left\Vert \left[  \left\vert x\right\vert -2\left\vert
\alpha\right\vert \left(  \frac{x^{-}}{\alpha}\wedge\frac{-x^{+}}{\alpha
}\right)  ^{\ast}\right]  ^{+}\right\Vert +\left\vert \alpha\right\vert
\]
Thus, $\left\vert \alpha\right\vert \leq\left\Vert x+\alpha\right\Vert _{1}$.
Also,%
\begin{align*}
\left\Vert x\right\Vert  &  =\left\Vert \left\vert x\right\vert \right\Vert
\leq\left\Vert \left[  \left\vert x\right\vert -2\left\vert \alpha\right\vert
\left(  \frac{x^{-}}{\alpha}\wedge\frac{-x^{+}}{\alpha}\right)  ^{\ast
}\right]  ^{+}\right\Vert +\left\Vert 2\left\vert \alpha\right\vert \left(
\frac{x^{-}}{\alpha}\wedge\frac{-x^{+}}{\alpha}\right)  ^{\ast}\right\Vert \\
&  \leq\left\Vert \left[  \left\vert x\right\vert -2\left\vert \alpha
\right\vert \left(  \frac{x^{-}}{\alpha}\wedge\frac{-x^{+}}{\alpha}\right)
^{\ast}\right]  ^{+}\right\Vert +2\left\vert \alpha\right\vert \leq2\left\Vert
x+\alpha\right\Vert _{1}%
\end{align*}
Therefore,%
\[
\left\Vert x\right\Vert +\left\vert \alpha\right\vert \leq3\left\Vert
x+\alpha\right\Vert _{1}%
\]
and the first inequality follows. Now, we write%
\begin{align*}
\left\Vert x+\alpha\right\Vert _{1}  &  =\left\Vert \left[  \left\vert
x\right\vert -2\left\vert \alpha\right\vert \left(  \frac{x^{-}}{\alpha}%
\wedge\frac{-x^{+}}{\alpha}\right)  ^{\ast}\right]  ^{+}\right\Vert
+\left\vert \alpha\right\vert \\
&  \leq\left\Vert x\right\Vert +2\left\vert \alpha\right\vert \left\Vert
\left(  \frac{x^{-}}{\alpha}\wedge\frac{-x^{+}}{\alpha}\right)  ^{\ast
}\right\Vert +\left\vert \alpha\right\vert \\
&  \leq\left\Vert x\right\Vert +3\left\vert \alpha\right\vert \leq3\left(
\left\Vert x\right\Vert +\left\vert \alpha\right\vert \right)  .
\end{align*}
This leads to the second inequality.
\end{proof}

It is well-known that if $E$ is a Banach space, then the formula%
\[
\left\Vert x+\alpha\right\Vert _{c}=\left\Vert x\right\Vert +\left\vert
\alpha\right\vert ,\text{ for all }x\in E\text{ and }\alpha\in\mathbb{R}%
\]
defines a norm on $E\oplus\mathbb{R}$ under which $E\oplus\mathbb{R}$ is a
Banach space \cite{BD1973}. This classical result is a key step in the proof
of the following theorem.

\begin{theorem}
\label{Banach}Let $E$ be truncated normed Riesz space and suppose that
$E\oplus\mathbb{R}$ is equipped with any unitization norm. Then, $E$ is a
Banach lattice if and only if $E\oplus\mathbb{R}$ is a Banach lattice.
\end{theorem}

\begin{proof}
Let $\left\Vert .\right\Vert _{u}$ the unitization norm on $E\oplus\mathbb{R}%
$. If $E\oplus\mathbb{R}$ is a Banach lattice then so is $E$. Indeed, being a
maximal ideal in $E\oplus\mathbb{R}$, $E$ is a closed set in $E\oplus
\mathbb{R}$ (see Corollary 3, Page 85 in \cite{S1974}). Conversely, assume
that $E$ is a Banach lattice. Then $E$ is a closed set in $E\oplus\mathbb{R}$.
By Theorem \ref{closed}, the norms $\left\Vert .\right\Vert _{u}$ and
$\left\Vert .\right\Vert _{1}$ are equivalent. Using Lemma \ref{classic}, we
infer that $\left\Vert .\right\Vert _{u}$ and $\left\Vert .\right\Vert _{c}$
are equivalent. But then $\left(  E\oplus\mathbb{R},\left\Vert .\right\Vert
_{u}\right)  $ is complete because so is $\left(  E\oplus\mathbb{R},\left\Vert
.\right\Vert _{c}\right)  $. This ends the proof of the corollary.
\end{proof}

Alluding to completeness, the following observation is worth mentioning. Let
$E$ be a truncated normed Riesz space. The norm completion $\overline{E}$ of
$E$ is Banach lattice with $E$ as a Riesz subspace \cite[Theorem 2.4]{AB2006}.
Applying the classical Birkhoff's Inequality \cite[Theorem 1.9]{AB2006} in the
Riesz space $E\oplus\mathbb{R}$, we derive that if $x,y\in E^{+}$ then%
\[
\left\vert x^{\ast}-y^{\ast}\right\vert =\left\vert 1\wedge x-1\wedge
y\right\vert \leq\left\vert x-y\right\vert .
\]
It follows that $\left\Vert x^{\ast}-y^{\ast}\right\Vert \leq\left\Vert
x-y\right\Vert $. We deduce that the truncation on $E$ is uniformly
continuous, so it extends uniquely to a truncation on $\overline{E}$ in such a
way that $\overline{E}$ becomes a truncated Banach lattice.

We end this paper with the following representation theorem which can be seen
as an extension of the famous Kakutani's representation theorem \cite[Theorem
3.6]{AA2002}. Recall that if $E$ is a Banach lattice with a strong unit $e>0$
then the gauge function given by%
\[
\left\Vert x\right\Vert _{\infty}=\inf\left\{  \lambda\in\left(
0,\infty\right)  :\left\vert x\right\vert <\lambda\right\}  ,\text{ for all
}x\in E
\]
makes $E$ into a unital $AM$-space which is, by the aforementioned Kakutani's
result, lattice isomorphic with $C\left(  K\right)  $ for some compact
Hausdorff space $K$. By the way, the Banach lattice of all real-valued
continuous functions on a locally compact Hausdorff space $X$ vanishing at
infinity is denoted by $C_{0}\left(  X\right)  $. We are in position now to
prove the last result of this paper. But recall first that $E_{\ast}=\left\{
x\in E:\left\vert x\right\vert ^{\ast}=\left\vert x\right\vert \right\}  $,
where $E$ is a truncated Riesz space.

\begin{corollary}
\label{C0}Let $E$ be a truncated Banach lattice such that if $x\in E^{+}$ then
$\mu x\in E_{\ast}$ for some $\mu\in\left(  0,\infty\right)  $. The formula%
\[
\left\Vert x\right\Vert _{\infty}=\inf\left\{  \lambda\in\left(
0,\infty\right)  :\frac{1}{\lambda}x\in E_{\ast}\right\}
\]
defines a lattice norm on $E$ under which $E$ is lattice isometric to
$C_{0}\left(  X\right)  $ for some locally compact Hausdorff space $X$.
\end{corollary}

\begin{proof}
Let $x\in E$ and $\alpha\in\mathbb{R}$. There exists $\mu\in\left(
0,\infty\right)  $ such that $\mu\left\vert x\right\vert \in E_{\ast}$ or,
equivalently, $\mu\left\vert x\right\vert \leq1$ in $E\oplus\mathbb{R}$. Thus,%
\[
\left\vert x+\alpha\right\vert \leq\left\vert x\right\vert +\left\vert
\alpha\right\vert =\frac{1}{\mu}\left\vert \mu x\right\vert +\left\vert
\alpha\right\vert \leq\frac{1}{\mu}+\left\vert \alpha\right\vert .
\]
It follows that $1$ is a strong unit in $E\oplus\mathbb{R}$. On the other
hand, $E\oplus\mathbb{R}$ is a Banach lattice with respect to any unitization
norm. Accordingly, the gauge function given by%
\[
\left\Vert x+\alpha\right\Vert _{\infty}=\inf\left\{  \lambda\in\left(
0,\infty\right)  :\left\vert x+\alpha\right\vert <\lambda\right\}  ,\text{ for
all }x\in E\text{ and }\alpha\in\mathbb{R}%
\]
defines a lattice norm on $E\oplus\mathbb{R}$ under which $E\oplus\mathbb{R}$
is a unital $AM$-space. We derive, by Kakutani's Representation Theorem, that
$\left(  E\oplus\mathbb{R},\left\Vert .\right\Vert _{\infty}\right)  $ is
lattice isometric to $C\left(  K\right)  $ for some compact Hausdorff space
$K$. Now, observe that if $x\in E$ then%
\[
\left\Vert x\right\Vert _{\infty}=\inf\left\{  \lambda\in\left(
0,\infty\right)  :\left\vert x\right\vert <\lambda\right\}  =\inf\left\{
\lambda\in\left(  0,\infty\right)  :\frac{1}{\lambda}x\in E_{\ast}\right\}  .
\]
On the other hand, since $E$ is a maximal order ideal in $E\oplus\mathbb{R}$
(see Proposition \ref{ideal}), there exists $k_{0}\in K$ such that $E$ is
lattice isomorphic with the maximal order ideal%
\[
M_{k_{0}}=\left\{  f\in C\left(  K\right)  :f\left(  k_{0}\right)  =0\right\}
\]
of $C\left(  K\right)  $ (see \cite[Example 27.7]{LZ1971}). Putting
$X=K\backslash\left\{  k_{0}\right\}  $, we derive that $X$ is locally compact
Hausdorff space \cite{E1989} and $E$ is lattice isomorphic with $C_{0}\left(
X\right)  $, completing the proof of the theorem.
\end{proof}

The last lines of the paper outlines how can Corollary \ref{C0} be considered
as a generalization of the Kakutani's representation theorem. For continuous
real-valued functions on a topological space, we refer the reader to
\cite{E1989,GJ1960}. Consider a Banach lattice $E$ with a strong unit $e>0$
and put%
\[
x^{\ast}=e\wedge x,\text{ for all }x\in E^{+}.
\]
Obviously, this equality makes $E$ into a truncated Banach lattice. Moreover,
it is readily seen that the extra condition of Corollary \ref{C0} is
fulfilled. It follows that $E$ is lattice isomorphic to $C_{0}\left(
X\right)  $ for some locally compact Hausdorff space. If $X$ is a compact, we
have nothing to prove. Assume that $X$ is not compact. Hence, $C_{0}\left(
X\right)  $ is lattice isomorphic with the maximal order ideal%
\[
M_{\infty}=\left\{  f\in C\left(  X^{\ast}\right)  :f\left(  \infty\right)
=0\right\}
\]
of $C\left(  \omega X\right)  $, where $\omega X$ is the one-point
compactification of $X$. Now, as $E$ has a strong unit, $M_{\infty}$ in turn
has a strong unit, say $u$. Obviously, $u\left(  x\right)  \neq0$ for all
$x\in X$. Thus, a lattice isomorphism $T:M_{\infty}\rightarrow C\left(  \omega
X\right)  $ can be defined by putting%
\[
T\left(  f\right)  =u^{-1}f,\text{ for all }f\in M_{\infty}.
\]
We derive that $E$ and $C\left(  \omega X\right)  $ are lattice isomorphic.
Finally, re-norming finally $E$ in a natural way, we infer that $E$ and
$C\left(  \omega X\right)  $ are lattice isometric.

\end{document}